\documentclass[12pt]{article}%
\usepackage{amsmath}
\usepackage{amsfonts}
\usepackage{amssymb}
\usepackage{graphicx}%
\providecommand{\U}[1]{\protect\rule{.1in}{.1in}}
\newtheorem{theorem}{Theorem}

\newtheorem{corollary}[theorem]{Corollary}

\newtheorem{definition}[theorem]{Definition}
\newtheorem{example}[theorem]{Example}

\newtheorem{lemma}[theorem]{Lemma}

\newtheorem{problem}[theorem]{Problem}
\newtheorem{proposition}[theorem]{Proposition}

\newenvironment{proof}[1][Proof]{\noindent\textbf{#1.} }{\ \rule{0.5em}{0.5em}}
\begin{document}

\title{Completeness for vector lattices}
\author{Youssef Azouzi\\{\small Research Laboratory of Algebra, Topology, Arithmetic, and Order}\\{\small Faculty of Mathematical, Physical and Natural Sciences of Tunis}\\{\small Tunis-El Manar University, 2092-El Manar, Tunisia}}
\date{}
\maketitle

\begin{abstract}
The notion of unboundedly order converges has been recieved recently a
particular attention by several authors. The main result of the present paper
shows that the notion is efficient and deserves that care. It states that a
vector lattice is universally complete if and only if it is unboundedly order
complete. Another notion of completeness

will be treated is the notion of sup-completion introduced by Donner.

\end{abstract}

\section{Introduction}

This paper deals with various notions of completeness for a vector lattice and
studies the connection between them. We will pay a particular attention to the
notion of unbounded order convergence and then to the unbounded order
completeness. A normed space is said to be complete if every Cauchy sequence
is norm convergent. In contrast to this definition, notions as Dedekind
completeness and universal completeness are often defined via the existence of
supremum for certain families. A Dedekind complete vector lattice $X$ is also
called order complete and this is very meaningful because the order
completeness is equivalent to the fact that order Cauchy nets are order
convergent. One can ask whether or not there is an analogous characterization
of universal completeness. In the recent paper \cite{L-174}, the authors
proved that under some extra condition a vector lattice is universal complete
if and only if every $uo$-Cauchy net is $uo$-convergent. The main result of
the present paper (Theorem \ref{Main}) states that this equivalence is still
true without any further assumption. This result shows how the notion of
unboundedly order convergence is both deep and natural.

Another notion of completeness, which concerns cones, will also be treated
here and connected to the previous notions: the notion of sup-completion. It
has been introduced by Donner in \cite{Donner} and used by Grobler in his
recent papers. It will play here a crucial role in the proof of our main theorem.

The structure of the paper is as follows. The next section (Section 2) is
devoted to the notion of sup-completion mentioned above. Every order complete
vector lattice $X$ has a sup-completion $X_{s}.$ Roughly speaking, $X_{s}$
plays the same role for $X$ as $\mathbb{R}_{\infty}=\mathbb{R}\cup\left\{
\infty\right\}  $ does for $\mathbb{R}$. In this section we provide the most
important properties of the cone $X_{s}$ and prove some new results. The most
useful for us is Theorem \ref{TA} which characterizes elements in $X_{s}$ that
do not belong to $X^{u},$ where $X^{u}$ is the universal completion of $X.$ In
section 3 we prove the main Theorem of this paper which states that a vector
lattice is universally complete if and only if it is $uo$-complete. This
section contains several applications of this result. In particular, it helps
to simplify some proofs in \cite{L-196} and \cite{L-197}. The last part of
this section is devoted to the $\sigma$-completeness.

\section{Sup-completion.}

Throughout this section, $X$ stands for an order complete vector lattice and
its universal completion is denoted by $X^{u}.$ We know from \cite{Donner}
that there exists a unique (up to isomorphisms) order complete cone, called
the \textit{sup-completion} of $X$ and denoted by $X_{s},$ with the following properties:

\begin{enumerate}
\item $X$ is the set of invertible elements in $X_{s}$ coinciding algebraic
and order structures$.$

\item For every $y\in X_{s}$ we have $y=\sup\left\{  x\in X:x\leq y\right\}
.$

\item We have $x+f\wedge y=\left(  x+f\right)  \wedge\left(  x+y\right)  $
whenever $x,y\in X_{s},f\in X$. If, in addition, $X$ has a weak order unit $e$
then $x=\sup\left(  ke\wedge x\right)  $ for all $0\leq x\in X_{s}.$

\item If $x\in X$ and $y\in X_{s}$ satisfy $y\leq x$ then $y\in X.$

\item $X_{s}$ has a biggest element.

\item For any two non-empty subsets $A,B\subset X_{s}$ satisfying $\sup A=\sup
B$ the equality
\[
\sup\limits_{a\in A}\left(  a\wedge x\right)  =\sup\limits_{b\in B}\left(
b\wedge x\right)
\]
holds for every $x\in X.$ In particular if $m=\sup A$ then%
\[
m\wedge x=\sup\limits_{a\in A}\left(  a\wedge x\right)  .
\]

\end{enumerate}

We shall throughout refer to these properties as P1, P2, etc. A cone
satisfying properties P1-P3) is said to be \textit{imbedded }for $X.$ An
imbedded cone is said to be \textit{tight }provided that it satisfies property
P4). We refer the reader to \cite{Donner} for further details on the subject.

It follows from the construction of the sup-completion, that if $A,B\subset X$
then $\sup\left(  A+B\right)  =\sup A+\sup B$ in $X_{s}$ (see the proof of
\cite[Theorem 1.4]{Donner}). For $x\in X_{s},$ we shall use the notation
$\left[  x\right]  ^{\leq}$ to denote the set $\left\{  y\in X:y\leq
x\right\}  .$ So Property 2) above becomes $x=\sup\left[  x\right]  ^{\leq}$
for all $x\in X_{s}.$ The following lemma tells us that $X_{s}$ has the Riesz
Decomposition Property.

\begin{lemma}
If $0\leq x,y,z\in X_{s}$ with $x\leq y+z$ then there exist $y_{1},z_{1}\in
X_{s}$ such that $y_{1}\leq y,$ $z_{1}\leq z$ and $x=y_{1}+z_{1}.$
\end{lemma}

\begin{proof}
Since $y+z=\left(  \sup\left[  y\right]  ^{\leq}+\sup\left[  z\right]  ^{\leq
}\right)  =x\wedge\sup\left(  \left[  y\right]  ^{\leq}+\left[  z\right]
^{\leq}\right)  $ we have%
\[
x=x\wedge\sup\left(  \left[  y\right]  ^{\leq}+\left[  z\right]  ^{\leq
}\right)  =\sup\limits_{a\in\left[  y\right]  ^{\leq},b\in\left[  z\right]
^{\leq}}x\wedge\left(  a+b\right)  .
\]
Now, for $\left(  a,b\right)  \in\left[  y\right]  ^{\leq}\times\left[
z\right]  ^{\leq}$ we define%
\[
x_{\left(  a,b\right)  }=x\wedge\left(  a+b\right)  ,y_{\left(  a,b\right)
}=x_{\left(  a,b\right)  }\wedge a,\text{ and }z_{\left(  a,b\right)
}=x_{\left(  a,b\right)  }-y_{\left(  a,b\right)  }.
\]
Then $x_{\left(  a,b\right)  }=y_{\left(  a,b\right)  }+z_{\left(  a,b\right)
,}$ $y_{\left(  a,b\right)  }\leq a\leq y$ and $z_{\left(  a,b\right)  }\leq
b\leq z.$ It is now clear that the elements $y_{1}=\sup y_{\left(  a,b\right)
}$ and $z_{1}=\sup z_{\left(  a,b\right)  }$ satisfy $y_{1}\leq y,$ $z_{1}\leq
z$ and $x=y_{1}+z_{1},$ which proves the result.
\end{proof}

It is also useful to observe that the following Birkhoff Inequality holds: If
$a,b\in X,c\in X_{s}$ then%
\[
\left\vert a\wedge c-b\wedge c\right\vert \leq\left\vert a-b\right\vert .
\]
Indeed, we have
\[
a\wedge c=\sup\limits_{x\in\left[  c\right]  ^{\leq}}a\wedge x,
\]
and, by the standard Birkhoff inequality,
\[
a\wedge x\leq b\wedge x+\left\vert a-b\right\vert \leq b\wedge c+\left\vert
a-b\right\vert ,\text{ }x\in C.
\]
Hence $a\wedge c\leq b\wedge c+\left\vert a-b\right\vert .$ Similarly we get
$b\wedge c\leq a\wedge c+\left\vert a-b\right\vert .$

\begin{definition}
\emph{Let }$X,Y$\emph{ be two order complete vector lattices with
sup-completions }$X_{s}$\emph{ and }$Y_{s}$\emph{, respectively. We say that
an increasing map} $g:X_{s}\longrightarrow Y_{s}$\emph{ is} \emph{(}%
left\emph{)} order continuous \emph{if }$g\left(  x_{\alpha}\right)  \uparrow
g\left(  x\right)  $\emph{ whenever} $x_{\alpha}\uparrow x.$
\end{definition}

\begin{proposition}
\label{Ext}Let $X$ and $Y$ be order complete vector lattices and let
$f:X\longrightarrow Y$ be an order continuous increasing map. Then $f$ can be
extended to a left order continuous increasing map from $X_{s}$ to $Y_{s}.$
Moreover, if $f$ is additive (resp. linear), then $g$ is additive (resp.
additive and positively homogeneous)
\end{proposition}

\begin{proof}
Define $g:X_{s}\longrightarrow Y_{s}$ by%
\[
g\left(  x\right)  =\sup\left\{  f\left(  y\right)  :y\in\left[  x\right]
^{\leq}\right\}  \in Y_{s}\qquad\text{for all }0\leq x\in X_{s}.
\]
Then $g\left(  x\right)  =\sup f\left(  x_{\alpha}\right)  $ for any net
$\left(  x_{\alpha}\right)  $ in $X$ such that $x_{\alpha}\uparrow x.$ Indeed
the inequality $\sup f\left(  x_{\alpha}\right)  \leq g\left(  x\right)  $
follows from the definition of $g.$ Now putting $a=\sup f\left(  x_{\alpha
}\right)  $ we see that for all $y\in\left[  x\right]  ^{\leq}$ we have
$y\wedge x_{\alpha}\uparrow y$ and $f\left(  x_{\alpha}\wedge y\right)  \leq
f\left(  x_{\alpha}\right)  \leq a.$ So $f\left(  y\right)  =\lim f\left(
y\wedge x_{\alpha}\right)  \leq a.$ Hence $g\left(  x\right)  =a=\sup f\left(
x_{\alpha}\right)  .$ Let us show that $g$ is order left continuous. If
$x=\sup x_{\alpha}$ in $X_{s},$ then $\sup g\left(  x_{\alpha}\right)
=\sup\limits_{\alpha}\sup\limits_{y\in\left[  x_{\alpha}\right]  ^{\leq}%
}f\left(  y\right)  =\sup\limits_{y\in%
{\textstyle\bigcup\limits_{\alpha}}
\left[  x_{\alpha}\right]  ^{\leq}}f\left(  y\right)  =g\left(  x\right)  $
because the set $%
{\textstyle\bigcup\limits_{\alpha}}
\left[  x_{\alpha}\right]  ^{\leq}$ defines an increasing net with supremum
$x.$

Assume now than $x\leq y$ in $X_{s}.$ Then $\left[  x\right]  ^{\leq}%
\subset\left[  y\right]  ^{\leq}$ and hence%
\[
g\left(  x\right)  =\sup\left\{  f\left(  t\right)  :t\in\left[  x\right]
^{\leq}\right\}  \leq\sup\left\{  f\left(  t\right)  :t\in\left[  y\right]
^{\leq}\right\}  =g\left(  x\right)  .
\]
If $f$ is additive we get%
\[
f\left(  \left[  x+y\right]  ^{\leq}\right)  =f\left(  \left[  x\right]
^{\leq}+\left[  y\right]  ^{\leq}\right)  =f\left(  \left[  x\right]  ^{\leq
}\right)  +f\left(  \left[  y\right]  ^{\leq}\right)  ,
\]
and hence $g\left(  x+y\right)  =g\left(  x\right)  +g\left(  y\right)  .$
Similarly we show that $g\left(  \lambda x\right)  =\lambda g\left(  x\right)
$ for $x\in X_{s},\lambda\in\lbrack0,\infty)$ if $f$ is linear.
\end{proof}

As an example we can extend any band projection $P$ to $X_{s}$. So, if $u\in
X_{+}$ we can define $P_{u}x$ for $0\leq x\in X_{s}$ by the formula%
\[
P_{u}x=\sup\left\{  P_{u}y:y\in X_{+},y\leq x\right\}  .
\]
On the other hand we wish to give a meaning to the expression $P_{u}x$ if
$u\in X_{s}.$ This notation appeared without explanation in \cite{L-21} and
will be justified below. Fix $0\leq a\in X_{s}$ and define a map $\pi_{a}$ on
$X$ by putting $\pi_{a}\left(  x\right)  =\sup\limits_{n\in\mathbb{N}}\left(
x\wedge na\right)  =\sup\limits_{\alpha>0}\left(  x\wedge\alpha a\right)  $
for a all $x\in X_{+}$ and then $\pi_{a}\left(  x\right)  =\pi_{a}\left(
x^{+}\right)  -\pi_{a}\left(  x^{-}\right)  $ for any $x\in X.$ It is easily
seen that $\pi_{\lambda a}\left(  x\right)  =\pi_{a}\left(  x\right)  $ for
all real $\lambda>0.$

We show next that $\pi_{a}$ is a band projection on $X.$

\begin{lemma}
Let $X$ be an order complete vector lattice with weak order unit $e$ and
$0\leq a\in X_{s}^{+}.$ Then $\pi_{a}=P_{\pi_{a}\left(  e\right)  }.$
\end{lemma}

\begin{proof}
Let $x\in X_{+}.$ Then
\begin{align*}
P_{\pi_{a}\left(  e\right)  }x  &  =\sup\limits_{k}\left(  x\wedge k\pi
_{a}\left(  e\right)  \right)  =\sup\limits_{k}\left(  x\wedge\sup
\limits_{\ell}ke\wedge kla\right) \\
&  =\sup\limits_{k}\sup\limits_{\ell}\left(  x\wedge ke\wedge kla\right)
=\sup\limits_{k}\left(  ke\wedge\sup\limits_{\ell}\left(  x\wedge kla\right)
\right) \\
&  =\sup\limits_{k}\left(  ke\wedge\pi_{ka}\left(  x\right)  \right)
=\sup\limits_{k}\left(  ke\wedge\pi_{a}\left(  x\right)  \right)  =\varphi
_{a}\left(  x\right)  .
\end{align*}
The last equality is true because $e$ is a weak order unit.
\end{proof}

From now on we shall use the notation $\pi_{a}=P_{a}$ for all $0\leq a\in
X_{s}.$ It follows easily from the definition that if $x_{\alpha}\uparrow x$
in $X_{s}$ then $P_{x_{\alpha}}\uparrow P_{x}.$

It is clear that Proposition \ref{Ext} is still true for functions defined on
$X_{+}.$ Thus one can define $x^{p}$ for all $0\leq x\in X_{s}$ and $p\geq1$
(one can also use \cite[Corollary 4.3]{L-06}). Now, using the notations
introduced above we can state a generalization of the Chebychev Inequality in
vector lattices (see \cite[Theorem 3.9]{AT}).

\begin{proposition}
Let $X$ be an order complete vector lattice with weak unit element $e$ and $T$
a conditional expectation with $Te=e.$ Let $x$ and $y$ be two positive
elements in $X_{s}$ with $y\in R\left(  T\right)  ,$ the range of $T.$ Then%
\[
y^{p}TP_{\left(  x-y\right)  ^{+}}e\leq Tx^{p}.
\]

\end{proposition}

\begin{proof}
According to \cite[Theorem 3.9]{AT} we have $y^{p}TP_{(u-y)^{+}}e\leq
Tu^{p}\leq Tx^{p}$ for all $u\in\left[  x\right]  ^{\leq}.$ The result follows
because $P_{(u-y)^{+}}e\uparrow P_{(x-y)^{+}}e$ as $u\uparrow x.$
\end{proof}

A Riemann integral on vector lattices was introduced in \cite{AR}, it is a
faithful generalization of the classical theory. The following result has been
proved in \cite{AR}.

\begin{lemma}
\label{R1}Let $p\in\left(  0,\infty\right)  $ and $a,\varepsilon\in\left(
0,\infty\right)  $ with $\varepsilon<a$. Then
\[
x^{p}-\varepsilon^{p}e=\int_{\varepsilon}^{a}pt^{p-1}P_{(x-te)^{+}%
}edt\text{\quad for all }x\in X\text{ with }\varepsilon e<x\leq ae.
\]

\end{lemma}

Now if $p\in\lbrack1,\infty),$ the same proof yields the following:

\begin{lemma}
\label{L10}Let $a\in\left(  0,\infty\right)  $ and $p\in\lbrack1,\infty).$
Then%
\[
x^{p}=\int_{0}^{a}pt^{p-1}P_{(x-te)^{+}}edt\text{\quad for all }x\in X\text{
with }0\leq x\leq ae.
\]

\end{lemma}

We need also an additional lemma.

\begin{lemma}
\label{L11}Let $X$ be an order complete vector lattice and $\left(  f_{\alpha
}\right)  $ a net of increasing maps from $\left[  a,b\right]  $ to $X.$ If
$f_{\alpha}\uparrow f$ then $\int_{a}^{b}f_{\alpha}\left(  t\right)
dt\uparrow\int_{a}^{b}f\left(  t\right)  dt.$
\end{lemma}

\begin{proof}
The map $f$ is increasing and it follows from \cite[Corollary 3]{AR} that $f$
and $f_{\alpha}$ are Riemann integrable for every $\alpha$. Moreover, consider
a sequence $\left(  \sigma_{p}\right)  $ of partitions of $\left[  a,b\right]
$ such that $\sigma_{p+1}$ is a refinement of $\sigma_{p},$ with $\sigma
_{p}=\left(  t_{p,0},....,t_{p,n_{p}}\right)  $ and assume that $\left\vert
\sigma_{p}\right\vert \longrightarrow0$ as $p\longrightarrow\infty,$ where
$\left\vert \sigma_{p}\right\vert $ denotes the mesh of the partition
$\sigma_{p}.$ Then%
\begin{align*}
\sup\limits_{\alpha}\int_{a}^{b}f_{\alpha}\left(  t\right)  dt  &
=\sup\limits_{\alpha}\sup\limits_{p}%
{\textstyle\sum\limits_{i}}
f_{\alpha}\left(  t_{p,i-1}\right)  \left(  t_{p,i}-t_{p,i-1}\right) \\
&  =\sup\limits_{p}\sup\limits_{\alpha}%
{\textstyle\sum\limits_{i}}
f_{\alpha}\left(  t_{p,i-1}\right)  \left(  t_{p,i}-t_{p,i-1}\right) \\
&  =\sup\limits_{p}L\left(  f,\sigma_{p}\right)  =\int_{a}^{b}f\left(
t\right)  dt.
\end{align*}
as required.
\end{proof}

If $f:[a,\infty)\longrightarrow X_{+}$ is Riemann integrable on every
subinterval $\left[  a,b\right]  ,$ $a\leq b<\infty,$ we define the
generalized Riemann integral as follows:%
\[
\int_{a}^{\infty}f\left(  t\right)  dt=\sup\limits_{a\leq b<\infty}\int
_{a}^{b}f\left(  t\right)  dt\in X_{s}.
\]

\begin{lemma}
\label{R2}Let $X$ be an order complete vector lattice with weak unit $e$ and
$x$ a positive vector in $X_{s}.$ Then%
\[
x=\int_{0}^{\infty}P_{\left(  x-te\right)  ^{+}}edt.
\]

\end{lemma}

\begin{proof}
Using Lemma \ref{L10} and Lemma \ref{L11} we have%
\begin{align*}
x  &  =\sup\limits_{k}x\wedge ke=\sup\limits_{k}\int_{0}^{\infty}P_{\left(
x\wedge ke-te\right)  ^{+}}edt\\
&  =\sup\limits_{k}\sup\limits_{a>0}\int_{0}^{a}P_{\left(  x\wedge
ke-te\right)  ^{+}}edt=\sup\limits_{a>0}\sup\limits_{k}\int_{0}^{a}P_{\left(
x\wedge ke-te\right)  ^{+}}edt\\
&  =\sup\limits_{a>0}\int_{0}^{a}P_{\left(  x-te\right)  ^{+}}edt=\int
_{0}^{\infty}P_{\left(  x-te\right)  ^{+}}edt.
\end{align*}

\end{proof}

\begin{lemma}
Let $X$ be an order complete vector lattice with weak order unit $e.$ If $x\in
X_{s}$ then $x\geq tP_{\left(  x-te\right)  ^{+}}e.$
\end{lemma}

\begin{proof}
Choose a net $\left(  x_{a}\right)  $ in $X_{+}$ such that $x_{\alpha}\uparrow
x.$ We have%
\begin{align*}
x  &  =\sup x_{\alpha}\geq\sup\limits_{\alpha}tP_{\left(  x_{\alpha
}-te\right)  ^{+}}e\\
&  =t\sup\limits_{\alpha}\sup\limits_{k}k\left(  x_{\alpha}-te\right)
^{+}\wedge e\\
&  =t\sup\limits_{k}\sup\limits_{\alpha}k\left(  x_{\alpha}-te\right)
^{+}\wedge e\\
&  =t\sup\limits_{k}k\left(  x-te\right)  ^{+}\wedge e=tP_{\left(
x-te\right)  ^{+}}e.
\end{align*}

\end{proof}

We are in position now to state the main Theorem of this section, which
characterizes elements in $X_{s}\setminus X^{u}.$

\begin{theorem}
\label{TA}Let $X$ be an order complete vector lattice and $0<e\in X$. Then the
following statements are equivalent

\begin{enumerate}
\item $x\in X_{s}\setminus X^{u}$

\item $\inf\limits_{\lambda\in\left(  0,\infty\right)  }P_{\left(  x-\lambda
e\right)  ^{+}}e>0.$
\end{enumerate}
\end{theorem}

\begin{proof}
We need only to prove that 1) implies 2). Assume then that $P_{\left(
x-\lambda e\right)  ^{+}}e\downarrow0$ as $\lambda\uparrow\infty.$ It follows
that the map $t\longmapsto q\left(  t\right)  =e-P_{\left(  x-\lambda
e\right)  ^{+}}e$ satisfies the conditions of Theorem 3.5 in \cite{ARd}
applied in the order complete vector lattice $I_{e},$ the ideal generated by
$e$. Then there is $y\in I_{e}^{u}\subset X^{u}$ such that $q\left(  t\right)
=e-P_{\left(  y-te\right)  ^{+}}e$ (here $T$ is the identity map). It follows
from Lemma \ref{R2} that
\[
y=\int_{0}^{\infty}P_{\left(  y-te\right)  ^{+}}edt=\int_{0}^{\infty
}P_{\left(  x-te\right)  ^{+}}edt=x,
\]
and then $x\in X^{u}.$
\end{proof}

A very useful consequence of the Theorem \ref{TA} is the following.

\begin{corollary}
\label{Car}Let $X$ be a vector lattice and $x\in X_{s}\setminus X^{u}.$ Then
there exists $a\in X$ such that $x\geq ta>0$ for all real $t>0.$
\end{corollary}

\begin{proof}
It follows from Theorem \ref{TA} that $\inf P_{\left(  x-te\right)  }e>0$ in
$X^{\delta},$ the order completion of $X.$ Since $X$ is order dense in
$X^{\delta},$ there is $a\in X_{+}$ such that $P_{\left(  x-te\right)  }e>a>0$
for all $t>0.$ It follows therefore that $x\geq tP_{\left(  x-te\right)
}e\geq ta$ for all $t>0.$
\end{proof}

The construction of the cone $X_{s}$ is a little complicated and it seems
interesting and helpful to know more about its relationship with the space
$X^{u}.$ An interesting information is given below.

\begin{proposition}
\label{XSU}Let $X$ be an order complete vector lattice then the positive cone
of $X^{u}$ is contained in $X_{s}.$
\end{proposition}

\begin{proof}
Fix a positive element $x\in X^{u}$ and choose an increasing net in $X$ such
that $x_{\alpha}\uparrow x.$ Let us define $y$ to be the supremum of the net
$\left(  x_{\alpha}\right)  $ in $X_{s}.$ We claim that $y=x.$ First, we show
that $y\in X^{u}.$ Otherwise, according to Corollary \ref{Car} there is an
element $a\in X$ such that $y\geq na>0$ for all integer $n.$ Now observe that
the following equality%
\[
\sup\left(  x_{\alpha}\wedge na\right)  =x\wedge na
\]
holds in $X^{u}$ and then in $X$ since $X$ is regular in $X^{u}.$ Moreover,
property P6) yields%
\[
na=na\wedge y=\sup\left(  x_{\alpha}\wedge na\right)  .
\]
Thus we get $x\geq na>0$ for all $n,$ which is impossible. So $y\in X^{u}$ as
claimed and then $y\geq x.$ Now if this inequality is strict then there is
$0<u\in X$ such that $y\geq x+u$ (because $X$ is order dense in $X^{u}$). It
follows that $y-u\geq x_{\alpha}$ for all $\alpha.$ But since $y-u\in X_{s}$
we get a again a contradiction, and the proof is finished.
\end{proof}

In the last part of this section we shall make some remarks about the natural
domain of a conditional expectation. Consider an order complete vector lattice
$X$ with weak unit $e.$ We recall that a conditional expectation operator is
an order continuous strictly positive projection on $X$, with its range
$R\left(  T\right)  $ is an order complete vector sublattice and with $Te=e.$
It was shown in \cite{L-24} that, for any conditional expectation $T$ on $X$,
there exists a largest vector sublattice of $X^{u}$ called the
\textit{natural} \textit{domain }of $T$, to which $T$ extends uniquely to a
conditional expectation. The domain of $T$ is denoted by $L^{1}\left(
T\right)  $ and defined, as in \cite{L-24}, in the following manner: Let $D$
be the set of elements $x$ in $X_{+}^{u}$ for which there is a net $\left(
x_{\alpha}\right)  $ in $X$ such that $x_{\alpha}\uparrow x$ and $\left(
Tx_{\alpha}\right)  $ is bounded in $X^{u}.$ Afterwards put $Tx=\sup
Tx_{\alpha}\in X^{u}$ for $x\in D$ and $L^{1}\left(  T\right)  =D-D.$ Of
course, it has been proved that $T$ is well defined on $D$ and can be extended
in a natural way to its domain. It is worth noting that the range of this
extended conditional expectation enjoys the nice property of being an
$f$-algebra \cite[Theorem 3.1]{AT}. This was also noticed by Grobler in
\cite{L-21} where he proved that if the conditional expectation $T$ with range
$\mathcal{F}$ is extended to its domain then its range is equal to
$\mathcal{F}^{u}$, which is actually an $f$-algebra. In this paper Grobler
suggested another approach to define the domain of $T$ using the
sup-completion $X_{s}$. Although it is not explicitly mentioned, it is
understood that this approach leads to the same definition introduced in
\cite{L-24}. The justifications given by Grobler for these facts are somehow
quick and there are some missing details (see \cite[Proposition 2.1]{L-21}).
We provide here an alternative proof. Following Grobler we say that an element
$0\leq x\in X_{s}$ is in the domain of $T$ if $Tx:=\sup Tx_{\alpha}\in X^{u}$
where $\left(  x_{\alpha}\right)  $ is any increasing net in $X$ such that
$x_{\alpha}\uparrow x.$ If $A$ denotes the set of such elements one can prove
that $A$ is contained in $X^{u}$ and then de set%
\[
\operatorname*{dom}T=A-A.
\]
It will be enough to prove that the set $A$ is in fact the set
$\operatorname*{dom}\left(  \tau\right)  $ defined in \cite{L-24} and denoted
above by $D.$ Assume by contradiction that $x\in A\setminus X^{u}.$ Then there
exists $u>0$ such that $x\geq nu$ for all $n\in\mathbb{N}$ and then $Tx\geq
nTu.$ So $Tu=0$ which contradicts the fact that $T$ is strictly positive. Now
inclusion $D\subset A$ follows from Proposition \ref{XSU}.

\section{Universal completion and uo-complete.}

We show in this section that a vector lattice $X$ is universally complete if
and only if it is uo-complete. This answers an open problem in \cite{L-174}
and offers another reason that the notion of unboundedly order convergence
deserves the particular attention that it recieved recently by many authors
(see \cite{L-65,L-196,L-197,L-174} and the refenrences cited there). Let first
recall some definitions. A net $\left(  x_{\alpha}\right)  _{\alpha\in A}$ in
a vector lattice $X$ is said to be \textit{order convergent} to $x$ (and we
write $x_{\alpha}\overset{o}{\longrightarrow}x$) if there exists a net
$\left(  y_{\beta}\right)  _{\beta\in B}$ such that $y_{\beta}\downarrow0$ and
for each $\beta\in B$ there exists $\alpha_{\beta}\in A$ satisfying
$\left\vert x_{\alpha}-x\right\vert \leq y_{\beta}$ for all $\alpha\geq
\alpha_{\beta}.$ A net $\left(  x_{\alpha}\right)  $ is said to
\textit{converge in unbounded order} to $x,$ and we write $x_{\alpha}%
\overset{uo}{\longrightarrow}x,$ if $\left\vert x_{\alpha}-x\right\vert \wedge
y\overset{o}{\longrightarrow}0$ for every $y\in X_{+}$ It should be noted that
$x_{\alpha}\overset{uo}{\longrightarrow}x$ if and only if $\left(  x_{\alpha
}\wedge b\right)  \vee a\overset{o}{\longrightarrow}\left(  x\wedge b\right)
\vee a$ for all $a\leq0\leq b.$ This can be reduced for positive nets to the
following%
\[
x_{\alpha}\wedge y\overset{o}{\longrightarrow}x\wedge y\text{ for all }y\in
X_{+}.
\]
A net $\left(  x_{\alpha}\right)  $ is said to be \textit{order Cauchy} if the
net $\left(  x_{\alpha}-x_{\beta}\right)  _{\left(  \alpha,\beta\right)  }$
converges in order to zero;.and it is \textit{uo-Cauchy }if the net $\left(
x_{\alpha}-x_{\beta}\right)  _{\left(  \alpha,\beta\right)  }$ uo-converges to zero.

\subsection{The main result}

In \cite{L-174} Li and Chen proved the following

\begin{proposition}
Let $X$ be an order complete vector lattice.

\begin{enumerate}
\item If $X$ is uo-complete then $X$ is universally complete.

\item If $X$ is universally complete and has the countable sup property then
$X$ is uo-complete.
\end{enumerate}
\end{proposition}

It is quite easy to prove, as we shall see, that the first assertion holds
without the assumption of order completeness. Moreover, the converse of 1)
also holds, showing that the converse is true without requiring any further
assumption on the space. Our main result of this section is the following:

\begin{theorem}
\label{Main}A vector lattice $X$ is universally complete if and only if is uo-complete.
\end{theorem}

\begin{proof}
Assume first that $X$ is uo-complete and consider an order Cauchy net $\left(
x_{\alpha}\right)  $ in $X.$ Then $\left(  x_{\alpha}\right)  $ is uo-Cauchy
and hence it is uo-convergent to some $x.$ But since $\left(  x_{\alpha
}\right)  _{\alpha\in A}$ is order Cauchy it has a bounded tail and then it
converges in order to $x.$ This shows that $X$ is order complete. We will
prove now that $X$ is laterally complete. This fact is proved in \cite{L-174}
as noted above and we provide here an alternative proof which involves the
sup-completion $X_{s}.$ Let $\left(  x_{\alpha}\right)  _{\alpha\in A}$ be a
family of mutually disjoint positive vectors in $X$ and define $y_{F}%
=\sup\limits_{\alpha\in F}x_{\alpha}$ for every finite subset $F$ of $A.$ Then
$x_{F}\uparrow x=\sup x_{\alpha}\in X_{s}$ and it is enough to show that $x\in
X.$ Observe that for every $u\in X_{+}$ we have that $x_{F}\wedge u\uparrow
x\wedge u\in X$ because $x\wedge u\leq u\in X.$ It follows that $\left(
x_{F}\right)  $ is uo-Cauchy in $X$ so $x_{F}$ is uo-convergent to some
$x^{\prime}\in X_{+}.$ But since $\left(  x_{F}\right)  $ is increasing,
$x^{\prime}=\sup x_{F}=x$ and we are done.

Conversely, assume that $X$ is universally complete and consider a uo-Cauchy
net $\left(  x_{\alpha}\right)  _{\alpha\in A}$ in $X.$ We have to show that
$\left(  x_{\alpha}\right)  $ is uo-convergent. By considering the nets
$\left(  x_{\alpha}^{+}\right)  $ and $\left(  x_{\alpha}^{-}\right)  $ we may
assume that $\left(  x_{\alpha}\right)  $ is a positive net. For each $a\in
X_{+},$ the net $\left(  x_{\alpha}\wedge a\right)  $ is order convergent to
some limit $\ell_{a}\in X_{+}.$We put $\ell=\sup\limits_{a\in X_{+}}\ell
_{a}\in X_{s}.$ The proof will be completed by two steps.

\textit{Step 1.} We will show first that%
\[
\ell_{a}=\ell\wedge a,\text{ }a\in X_{+}.
\]

Assume by contradiction that $x_{\alpha}\wedge a\longrightarrow\ell_{a}%
<\ell\wedge a$ for some $a\in X_{+}.$ Using property P6) of $X_{s}$ (see
Section 2) we know that%
\[
\ell\wedge a=\sup\limits_{u\in X_{+}}\ell_{u}\wedge a=\sup\limits_{u\in X_{+}%
}\left(  \ell_{u}\wedge a\right)  .
\]
Hence there exists $b\geq0$ such that%
\[
\ell_{b}\wedge a>\ell_{b}\wedge a\wedge\ell_{a}=\ell_{b}\wedge\ell_{a}.
\]
Now we have on one hand%
\[
x_{\alpha}\wedge b\wedge a=\left(  x_{\alpha}\wedge b\right)  \wedge
a\overset{o}{\longrightarrow}\ell_{b}\wedge a,
\]
and on the other hand%
\[
x_{\alpha}\wedge b\wedge a=\left(  x_{\alpha}\wedge b\right)  \wedge\left(
x_{\alpha}\wedge a\right)  \overset{o}{\longrightarrow}\ell_{b}\wedge\ell
_{a},
\]
which leads to the contradiction%
\[
\ell_{a}\wedge\ell_{b}=\ell_{b}\wedge a.
\]

\textit{Step 2.} We claim that $\ell\in X.$ We argue again by contradiction
and we suppose that $\ell\in X_{s}\diagdown X$ Pick an element $e$ in $X$ with
$e>0.$ Then by Theorem \ref{TA}, $u:=\inf\limits_{t>0}P_{\left(
\ell-te\right)  ^{+}}e>0$. Thus%
\[
tu=P_{\left(  \ell-te\right)  ^{+}}tu\leq P_{\left(  \ell-tu\right)  ^{+}%
}ta\leq\ell.
\]
Now, put $L=\sup nu\in X_{s}$ and observe that $x_{\alpha}\wedge nu\uparrow
x_{\alpha}\wedge L$ and that%
\[
\left(  x_{\alpha}\wedge L\right)  \wedge y=\left(  x_{\alpha}\wedge y\right)
\wedge L\overset{o}{\longrightarrow}\left(  \ell\wedge y\right)  \wedge
L=y\wedge L.
\]
Hence, by considering the net $\left(  x_{\alpha}\wedge L\right)  $ instead of
$\left(  x_{\alpha}\right)  $ and $L$ instead of $\ell,$ we may assume that
$x_{\alpha}\leq\ell=\sup nu$. In particular $x_{\alpha}\wedge nu\overset
{o}{\longrightarrow}x_{\alpha}$ for every $\alpha\in A.$ For fixed $\alpha\in
A$ we can find a net $\left(  u_{\gamma}\right)  _{\gamma\in\Gamma}$ with
$u_{\gamma}\downarrow0$ and such that for every $\gamma\in\Gamma$ there is
$n_{\gamma}$ satisfying:%
\[
n\geq n_{0}\leq\left\vert x_{\alpha}\wedge nu-x_{\alpha}\right\vert \leq
u_{\gamma}.
\]
Pick $\gamma$ in $\Gamma$ and choose $n_{\gamma}$ as above. Then we have for
all $n\geq n_{\gamma},$%
\begin{align*}
x_{\beta}-x_{\alpha}  &  =x_{\beta}-x_{\beta}\wedge\left(  n+1\right)
u+x_{\beta}\wedge\left(  n+1\right)  u-x_{\alpha}\wedge nu+x_{\alpha}\wedge
nu-x_{\alpha}\\
&  \geq x_{\beta}\wedge\left(  n+1\right)  u-x_{\alpha}\wedge nu-\left\vert
\ell\wedge nu-x_{\alpha}\right\vert \\
&  \geq x_{\beta}\wedge\left(  n+1\right)  u-nu-u_{\gamma}.
\end{align*}
It follows that%
\begin{align*}
\sup\limits_{\beta,\beta^{\prime}\geq\alpha}\left\vert x_{\beta}%
-x_{\beta^{\prime}}\right\vert \wedge u  &  \geq\lim\limits_{\beta}\left(
x_{\beta}\wedge\left(  n+1\right)  u-nu-u_{\gamma}\right)  \wedge u\\
&  =\left(  u-u_{\gamma}\right)  \wedge u.
\end{align*}
Since this holds for every $\gamma$ in $\Gamma$ we conclude that%
\[
\sup\limits_{\beta,\beta^{\prime}\geq\alpha}\left\vert x_{\beta}%
-x_{\beta^{\prime}}\right\vert \wedge u\geq u.
\]
But this occurs for every $\alpha\in A,$ which contradicts the fact that the
net $\left(  x_{\alpha}\right)  $ is uo-Cauchy and completes the proof.
\end{proof}

\subsection{Some applications}

Next, we give some applications of Theorem \ref{Main}. The first result will
be useful to develop short proofs of some of the results of
\cite{L-65,L-196,L-197}.

\begin{theorem}
\label{TB} \textit{Let} $X$ \textit{be an order continuous Banach lattice with
dual }$X^{\ast}.$

\begin{enumerate}
\item[\textrm{(i)}] If $\left(  \varphi_{\alpha}\right)  $ is a uo-Cauchy net
in $X^{\ast}$ that converges \textit{in the weak}${}^{\ast}$ \textit{topology}
$\sigma(X^{\ast},X)$ is uo-convergent to the same limit.

\item[(ii)] If $\left(  x_{\alpha}\right)  $ is a uo-Cauchy net in $X$ that
converges weakly then it is uo-convergent to the same limit.
\end{enumerate}
\end{theorem}

\begin{proof}
(i) As usual, we may assume that $\left(  \varphi_{\alpha}\right)  $ is
positive. By Theorem \ref{Main}, $\varphi_{\alpha}\overset{uo}{\longrightarrow
}\psi$ in $\left(  X^{\ast}\right)  ^{u}$ for some $\psi\in\left(  X^{\ast
}\right)  ^{u}$ and it will be enough to establish that $\psi=\varphi.$ To
this end, in view of property P4) of $X_{s}$ (see Section 2), we need only to
show that $\psi\wedge h=\varphi$ for all $h\in X^{\ast}$ with $h\geq\varphi.$
Now, pick $h$ in $X^{\ast}$ with $h\geq\varphi.$ Since $\varphi_{n}%
\overset{w^{\ast}}{\longrightarrow}\varphi$ it is easily seen that%
\[
\varphi_{\alpha}\wedge h\overset{w^{\ast}}{\longrightarrow}\varphi\wedge
h=\varphi.
\]
On the other hand $\varphi_{n}\wedge h\overset{o}{\longrightarrow}\psi\wedge
h$ and by \cite[Theorem 2.1]{L-196} $\varphi_{\alpha}\wedge h\overset{w^{\ast
}}{\longrightarrow}\psi\wedge h.$ Hence $\psi\wedge h=\varphi$ as required.

(ii) Consider a net $\left(  x_{\alpha}\right)  $ with $x_{\alpha}\overset
{w}{\longrightarrow}x$ and assule as usual that $\left(  x_{\alpha}\right)  $
is positive. Using Theorem \ref{Main} we know that $x_{\alpha}\overset
{uo}{\longrightarrow}z$ for some $z\in X^{u}.$ Let $y\in X^{+},$ then
$x_{\alpha}\wedge y\overset{o}{\longrightarrow}z\wedge y$ in $X^{u}$ and then
in $X$ (since $X$ is regular in $X^{u}$ and order complete). Using the fact
that $X$ is order continuous we get $x_{\alpha}\wedge y\overset{\left\Vert
{}\right\Vert }{\longrightarrow}z\wedge y$ and so $x_{\alpha}\wedge
y\overset{w}{\longrightarrow}z\wedge y.$ On the other hand in view of
$x_{\alpha}\overset{w}{\longrightarrow}x,$ it follows easily that $x_{\alpha
}\wedge y\overset{w}{\longrightarrow}x\wedge y$ for every $y\in X_{+}.$ Hence
$x\wedge y=z\wedge y$ for all $y\in X_{+},$ which gives the equality $x=z$ and
completes the proof.
\end{proof}

As an immediate consequence of Theorem \ref{TB} we get the following results.

\begin{theorem}
\cite[Theorem 2.2]{L-196} \textit{Let} $X$ \textit{be an order continuous
Banach lattice. Then any norm bounded uo-Cauchy net in} $X^{\ast}$
\textit{converges }in uo \textit{and} $|\sigma|(X^{\ast},X)$ \textit{to the
same limit}.
\end{theorem}

\begin{theorem}
\cite[Theorem 4.3]{L-197} Let $X$ be an order continuous Banach lattice. Then
every relatively weakly compact uo-Cauchy net converges uo-and $\left\vert
\sigma\right\vert \left(  X,X^{\ast}\right)  $ to the same limit.
\end{theorem}

\begin{proof}
If $\left(  x_{\alpha}\right)  $ is relatively weakly compact then it has a
subnet $\left(  x_{\beta}\right)  $ weakly convergent to some $x\in X.$
According to Theorem \ref{TB} $x_{\beta}\overset{uo}{\longrightarrow}x$ and
then $x_{\alpha}\overset{uo}{\longrightarrow}x$ because $\left(  x_{\alpha
}\right)  $ is uo-convergent in $X^{u}.$
\end{proof}

\begin{theorem}
\cite[Theorem 4.7]{L-197} Let $X$ be an order continuous Banach lattice. The
following statements are equivalent.

\begin{enumerate}
\item $X$ is a KB-space.

\item $X$ is boundedly uo-complete.

\item $X$ is sequentially uo-complete.
\end{enumerate}
\end{theorem}

\begin{proof}
We need only to prove that (1) $\implies$(2). Let $\left(  x_{\alpha}\right)
$ be a norm bounded uo-Cauchy net. By considering $\left(  x_{\alpha}^{\pm
}\right)  $ we may assume that $\left(  x_{\alpha}\right)  $ is positive. In
view of Theorem \ref{Main} we know that $x_{\alpha}\overset{uo}%
{\longrightarrow}x$ for some $x\in X^{u}.$ Since $X$ is order complete and
regular in $X^{u}$ it follows from \cite[Theorem 3.2]{L-65} that $x_{\alpha
}\wedge a\overset{o}{\longrightarrow}x\wedge a$ in $X$ for all $a\in X.$ Now
observe that $\left\Vert x\wedge a\right\Vert \leq\sup\left\Vert x_{\alpha
}\right\Vert <\infty.$ Hence the net $\left(  x\wedge a\right)  _{a\in X_{+}}$
is norm bounded and increasing in the KB-space $X.$ It follows that
$x=\sup\limits_{a\in X}x\wedge a\in X$ and we are done.
\end{proof}

Recall that a Banach lattice is said to satisfy the weak Fatou property if
there is a real $r>0$ such that for every increasing net $\left(  x_{\alpha
}\right)  $ with the supremum $x\in X$ it follows that $\left\Vert
x\right\Vert \leq r\sup\limits_{\alpha}\left\Vert x_{\alpha}\right\Vert .$
Since there is a Banach lattice $X$ with weak Fatou property such that
$X_{n}^{\widetilde{}}=\left\{  0\right\}  $ the next theorem improves
\cite[Theorem 2.3]{L-174}. Theorem \ref{Main} will be used again to get a
short proof.

\begin{theorem}
\textit{Let} $(x_{\alpha})$ \textit{be a norm bounded positive increasing net
in a Banach lattice }$X$\textit{. Suppose that }$X$ satisfies one of the
following conditions: (i) $X_{n}^{\sim}$ \textit{separates points of} $X;$
(ii) $X$ satisfies the weak Fatou property. \textit{Then} $(x_{\alpha})$
\textit{is uo-Cauchy in} $X$.
\end{theorem}

\begin{proof}
Let $x=\sup x_{\alpha}\in X_{s}.$ It is sufficient to show that $x\in X^{u}.$
Assume by contradiction that it is not the case, then by Corollary \ref{Car}
that there exists an element $a\in X$ such that $x\geq ta>0$ for all real
$t>0.$ It follows that
\begin{equation}
x_{\alpha}\wedge na\uparrow na\text{ in }X. \label{E4}%
\end{equation}
If (i) is satisfied, then there is $0\leq\varphi\in X_{n}^{\widetilde{}}$ such
that $\varphi\left(  a\right)  >0.$ It follows from $\varphi\left(  x_{\alpha
}\wedge na\right)  \uparrow n\varphi\left(  a\right)  $ that $n\varphi\left(
a\right)  \leq M\left\Vert \varphi\right\Vert $ for all $n,$ which is
impossible.\newline If ii) is satisfied then in virtu of \ref{E4} we obtain
$n\left\Vert a\right\Vert \leq C\sup\left\Vert x_{\alpha}\right\Vert $ for all
$n,$ which is also impossible. This completes the proof.
\end{proof}

\subsection{The $\sigma$-completeness}

It was proved \cite[Corollary 3.12]{L-65} that a uo-null sequence in $X$ is
o-null $X^{u}$. A concrete example in \cite[Example 2.6]{L-174} shows that
this result fails for nets. We can ask the following questions

\begin{problem}
Suppose that $X$ is an arbitrary order complete but not laterally complete
vector lattice. Is there a uo-Cauchy net in $X$ that fails to be order
convergent in $X^{u}$?
\end{problem}

\begin{problem}
Is it true that if $X$ is a vector lattice for which order convergence and
uo-convergence agree for sequences then $X$ is universally $\sigma$-complete?
\end{problem}

So far, we do not have an answer to the first question. However, the answer to
the second question is negative as the next example shows.

\begin{example}
Consider the space $X$ of real sequences $x=\left(  x_{n}\right)  _{n\geq1}$
such that $x_{2n}=x_{2n+1}$ for $n$ enough large. Then $X$ is a vector lattice
which is not laterally $\sigma$-complete. Indeed the sequence $\left(
y^{n}\right)  $ define by
\[
y^{n}=\left\{
\begin{array}
[c]{cc}%
e^{n} & \text{if }n\text{ is even}\\
0 & \text{if }n\text{ is odd}%
\end{array}
\right.
\]
has no supremum in $X,$ where $e^{n}$ is the sequence defined by $e_{k}^{n}=1$
if $k=n$ and $e_{k}^{n}=0$ otherwise. However, it is easy to see that a
sequence $\left(  y^{n}\right)  $ in $X$ is uo-null iff it is o-null.
\end{example}

In spite of the last example which give a negative answer to our question we
have the following `positive' result.

\begin{proposition}
For an order $\sigma$-complete vector lattice $X$ the following are equivalent

\begin{enumerate}
\item $X$ is universally $\sigma$-complete;

\item A sequence in $X$ is uo-null if only if it is o-null.
\end{enumerate}
\end{proposition}

The implication i) $\implies$ ii) follows from \cite[Corollary 3.12]{L-65}.
Assume conversely that ii) is satisfied and consider a disjoint positive
sequence in $X.$ According to \cite[Corollary 3.6]{L-65} $x_{n}\overset
{uo}{\longrightarrow}0$ and then by ii) $x_{n}\overset{o}{\longrightarrow}0.$
So $\left(  x_{n}\right)  $ is bounded. This shows that the sequence
$y_{n}=x_{1}\vee...\vee x_{n}$ is also bounded and then has a supremum as $X$
is order $\sigma$-complete. This proves that $X$ is laterally $\sigma
$-complete and ends the proof.

It has been proved in \cite[Lemma 2.11]{L-65} that if $Y$ is a regular order
complete sublattice of a vector lattice $X$ then every net in $Y$ that is
order convergent to some $x$ in $X$ is order convergent in $Y$ to the same
limit. Similarly we can prove the following.

\begin{lemma}
\label{LA}Let $Y$ be a vector sublattice of a vector lattice $X.$ Assume that
$Y$ is regular in $X$ and order $\sigma$-complete. If $\left(  y_{n}\right)
\subset Y$ and $y_{n}\overset{o}{\longrightarrow x}$ in $X$ then $x\in Y$ and
$y_{n}\overset{o}{\longrightarrow}x$ in $Y.$
\end{lemma}

The next Theorem is a "variant" of Theorem \ref{Main}.

\begin{theorem}
The vector lattice $X$ is universally $\sigma$-complete if and only if is
sequentially uo-complete.
\end{theorem}

\begin{proof}
It follows from \cite[Theorem 3.10]{L-65} that every universally $\sigma
$-complete vector lattice is sequentially uo-complete. Now assume that $X$ is
sequentially uo-complete. If $\left(  x_{n}\right)  $ is bounded above then
the sequence $\left(  y_{n}=x_{1}\vee...\vee x_{n}\right)  $ converges in
$X^{u},$ say $y_{n}\uparrow x\in X^{u}.$ It follows then that $\left(
y_{n}\right)  $ is o-Cauchy so it is uo-Cauchy in $X^{u}.$ Since $X$ is
regular in $X^{u}$ it follows from \cite[Theorem 3.2]{L-65} that $\left(
y_{n}\right)  $ is uo-Cauchy in $X$ and so it is uo-convergent in $X$ to some
$y\in X.$ It is easy now to see that $y=x=\sup x_{n}.$ Assume now that
$\left(  x_{n}\right)  $ is a sequence of positive disjoint elements in $X.$
Let $y_{n}=x_{1}\vee...\vee x_{n}.$ Then $y_{n}\uparrow y=\sup x_{n}\in
X^{u}.$ By the same way we prove that $y\in X.$
\end{proof}

\end{document}